\documentclass{article}

\usepackage[utf8]{inputenc}
\usepackage[english]{babel}

\usepackage{graphicx}
\usepackage{amssymb}
\usepackage{amsmath}
\usepackage{tikz-cd}
\usepackage{enumerate}

\usepackage{hyperref}

\newenvironment{proof}[1][Proof:]{\begin{trivlist}
\item[\hskip \labelsep {\bfseries #1}]}{\end{trivlist}}

\newtheorem{theorem}{Theorem}[section]
\newtheorem{corollary}[theorem]{Corollary}
\newtheorem{lemma}[theorem]{Lemma}
\newtheorem{proposition}[theorem]{Proposition}

\begin{document}

\title{Stability and periodicity in the modular representation theory of symmetric groups}

\author{Nate Harman}

\maketitle

\begin{abstract}
We study asymptotic properties of the modular representation theory of symmetric groups and investigate modular analogs of stabilization phenomena in characteristic zero. The main results are equivalences of categories between certain abelian subcategories of representations of $S_n$ and $S_m$ for different $n$ and $m$.  We apply these results to obtain a structural result for $FI$-modules, and to prove a result conjectured by Deligne in a recent letter to Ostrik.

\end{abstract}

\begin{section}*{Introduction}

For a partition $\lambda = (\lambda_1,\lambda_2,...,\lambda_k)$ and $n \gg |\lambda|$ let $\lambda(n) := (n-|\lambda|, \lambda_1,\lambda_2,...,\lambda_k)$ be the partition whose Young diagram is obtained from the young diagram for $\lambda$ by adding a long first row to make it have $n$ boxes total. Recently, a lot of work has been done in studying the representation theory of symmetric groups $S_n$ uniformly in $n$ by identifying various representations corresponding to $\lambda(n)$ for different values of $n$.

We'd like to identify three different (but certainly not mutually exclusive) types of results within this theory being developed:

\medskip

\noindent \emph{1. Results relating (sub)categories of representations of different symmetric groups $S_n$ to one another, and using this to understand which representation theoretic properties we should expect to stabilize, become periodic, grow polynomially, etc., as $n$ grows.}

In characteristic $0$ this is best exemplified by Deligne's construction \cite{Deligne} of categories $\underline{\text{Rep}}(S_t)$, which interpolate the categories of representations of $S_n$ in a strong sense (as rigid symmetric tensor categories).  Roughly speaking, the very existence of such a family forces the stabilization or polynomiality of a number of representation theoretic quantities.  Also in this direction is Sam and Snowden's detailed description (\cite{SS1} section 4) of their category $\text{Mod}_A$ which is closely related to the category of $FI$-modules (see \cite{SS2}). 

\medskip

\noindent \emph{2.  Results finding sequences $V_n$ of representations of symmetric groups $S_n$ in areas of math outside of representation theory which stabilize in this sense, and understanding the mechanisms this stability or polynomiality to occur.}

A huge amount of work has been done recently in this direction following the seminal paper by Church and Farb \cite{CF} introducing the notion of representation stability.  Two notable approaches to this are the theory of $FI$-modules developed by Church, Ellenberg, and Farb in \cite{CEF}, and Putman's notion of central stability from \cite{Putman}.  In particular, much of this theory has been developed over fields of arbitrary characteristic and in some cases over arbitrary noetherian rings (see \cite{CEFN}).

\medskip

\noindent \emph{3.   Results which revisit potentially difficult representation theoretic questions that become more tractable in the setting of representation stability.}

Many of the applications of $FI$-modules (e.g. \cite{CEF}) and in particular the bounds on the stability degree they give reduce the potentially computationally infeasible task of decomposing a cohomology space $H^i(X_n)$ as a representation for large $n$ into a reasonable one by reducing it to a calculation for much smaller value of $n$.  For a more representation theoretic result of this type (but for wreath products) see \cite{Harman} where we solve (and generalize) a conjecture of Marin by first answering the question in the stable setting and then passing down to the classical setting via the machinery of Deligne's categories and their generalizations. 

\medskip

In positive characteristic, Putman's central stability \cite{Putman} and the development of the theory of $FI$-modules over noetherian rings in \cite{CEFN} are pretty satisfactory results of type $2$ in the above trichotomy.  However they lack some of the nice descriptive properties of the characteristic $0$ theory coming from type $1$ results in that setting.

 The purpose of this paper is to develop the type $1$ theory in characteristic $p$.  We'd first like to mention a couple of known results in this direction:
 
 \medskip
 \begin{itemize}
\item In the 70's James gave an explicit description of decomposition multiplicities for Specht modules of the form $S^{(n-m,m)}$ into irreducible composition factors (\cite{James} Theorem 24.14).  One easy observation is that this decomposition really only depends on the last $\ell$ base $p$ digits of $n$ where $\ell$ is the number of base $p$ digits in $m$.  In particular if we fix $m$ and let $n$ grow these decompositions become periodic in $n$ with period $p^\ell$.

\item Recently, Nagpal looked at sequences of representations $V_n$ of $S_n$ coming from finitely generated $FI$-modules and showed that the cohomology groups $H^i(S_n,V_n)$ eventually become periodic in $n$ with period a power of $p$ \cite{Nagpal}.
\end{itemize}

These help motivate the central thesis of this paper: \emph{Many aspects of the modular representation theory of symmetric groups $S_n$ becomes periodic in $n$ with period a power of the characteristic.}  The paper is structured as follows:

\begin{itemize}
\item In section $1$ we fix some notation and go over some background results on a number of different topics we will be using later on. It will be largely expository.

\item Section $2$ contains the main results of the paper, a series of stability and periodicity results about subcategories of representations of symmetric groups in positive characteristic.

\item In section $3$ we give some applications of the theory we have developed. This includes results on numerical representation theoretic invariants, a structural theorem for $FI$-modules, and a conjecture of Deligne related to pre-Tannakian categories.
\end{itemize}


\medskip

\noindent \textbf{Remark:}  This line of inquiry was inspired by a recent letter from Deligne to Ostrik in which he shows the existence of a pre-Tannakian category in characteristic $p$ with an object of superexponential growth.  The construction considers the sequence of categories of representations of $S_n$ over $\mathbb{F}_p$, but avoids the need for stabilization by using an ultrafilter. Deligne conjectured that resulting category should only depend the choice of ultrafilter very mildly and there should only be one (up to equivalence) for each $p$-adic integer.  We revisit this in more detail in section \ref{Deligne}.

\end{section}

\section*{Acknowledgments}

Thanks to Alexander Kleshchev for spotting an error in an earlier version of the paper, and for sharing his many ideas. Thanks to Victor Ostrik for sharing the letter he received from Deligne which inspired this work, as well as for meeting with me to discuss this and related work. Thanks to Pierre Deligne for his helpful comments and feedback.  Thanks to, Pavel Etingof, Inna Entova-Aizenbud, and Seth Shelley-Abrahamson for many helpful conversations and comments. This material is based upon work supported by the National Science Foundation Graduate Research Fellowship under Grant No. 1122374.

\begin{section}{Notation, background, and miscellaneous facts}

This first section will be to fix some notation and introduce some of the relevant background material and machinery we will be using later on. For the most part it consists of known results for which we have left the proofs to the references, but it also includes (with proof) a few minor corollaries or reformulations of these results which we have created for the purposes of this paper.

\begin{subsection}{Some modular representations of symmetric groups}\label{rev}

For $\lambda = (\lambda_1,\lambda_2,...,\lambda_k)$ a partition of $n$ we define $M(\lambda)$ as the permutation module for the action of $S_n$ on the cosets of $S_{\lambda_1} \times S_{\lambda_2} \times... \times S_{\lambda_k}$.  Equivalently, this is the linearization of the natural action of $S_n$ on the collection of set-partitions $[n] = A_1 \cup A_2 \cup \dots \cup A_k$ with $|A_i| = \lambda_i$.  

Contained in $M(\lambda)$  is the Specht module $S^\lambda$ which is characterized by being the intersection of the kernels of all $S_n$-equivariant maps from $M(\lambda)$ to $M(\mu)$ with $\mu > \lambda$ in the dominance order.  Similarly there is the dual Specht module $S_\lambda$ which is a quotient of $M(\lambda)$ by the images of all $S_n$-maps from $M(\mu)$ with $\mu > \lambda$ in the dominance order. 

In characteristic $0$, all of these maps split, so the Specht and dual Specht modules coincide and moreover they form a complete list of isomorphism classes of irreducible representations of $S_n$.  

In characteristic $p$, the situation is somewhat different. Let $S^{\lambda \perp}$ be the submodule of the permutation module $M(\lambda)$ is spanned by the images of all $S_n$ equivariant maps from  $M(\mu)$ to $M(\lambda)$ with $\mu > \lambda$ in the dominance order, and say that a partition is $p$-regular if it does not contain $p$ parts of the same size. We have the following classification of irreducible representations of $S_n$ over an algebraically closed field of characteristic $p$.

\begin{theorem} {\textbf{(\cite{James} section 4)}}\label{modreview}

\begin{enumerate}

\item The quotient $D^\lambda := S^\lambda / (S^\lambda \cap S^{\lambda \perp})$ is either irreducible or zero.

\item $D^\lambda$ is nonzero iff $\lambda$ is $p$-regular.

\item If we let $\lambda$ run over the collection of $p$-regular partitions, the $D^\lambda$ form a complete list of the isomorphism classes of irreducible representations of $S_n$.

\end{enumerate}

\end{theorem}

We could also define the irreducible representations in terms of $p$-restricted partitions (i.e. those conjugate to a $p$-regular one) by taking:

$$D_\lambda := \text{soc}(S^\lambda)$$  

\noindent for $p$-restricted $\lambda$ (see \cite{Hemmer2}). These are related to our other description of irreducibles by the formula:

$$D_\lambda = D^{\lambda'} \otimes \text{sgn}$$

\noindent which follows from a similar formula for Specht and dual Specht modules.  Partitions $\lambda(n)$ are never $p$-restricted for $n \gg 0$, so for the most part the first description is better suited for our purposes.  However at some point we will want to use Schur-Weyl duality (described in section \ref{Schursect}), for which this second description is better behaved. 

\medskip

In addition to looking at the irreducible composition factors of permutation modules $M(\lambda)$, we will also be interested in their decomposition into indecomposable factors. The following theorem (see \cite{Martin} section 4.6, for example) summarizes what happens:

\begin{theorem}\label{young}\

\begin{enumerate}
\item There is a unique indecomposable direct summand of $M(\lambda)$ containing the Specht module $S^\lambda$, it is called the Young module $Y(\lambda)$.

\item  $Y(\lambda)$ is self dual and hence can also be characterized as the unique indecomposable direct summand of $M(\lambda)$ such that the quotient map form $M(\lambda)$ to the dual Specht module $S_\lambda$ factors through projection onto $Y(\lambda)$.

\item Any other indecomposable direct summand of $M(\lambda)$ is isomorphic to $Y(\mu)$ for some $\mu > \lambda$ in the dominance order.

\end{enumerate}
\end{theorem}

\end{subsection}

\begin{subsection}{Highest weight categories and Ringel duality}

We'll now recall some facts about highest weight categories. For detailed references we refer to \cite{CPS} and the appendix of \cite{DonkinBook}, but a quick reference we will recommend the worksheet \cite{Losev} by Losev, which gives a series of exercises developing the theory.

 We'd also like to give a disclaimer before going on that the category of representations of the symmetric group $S_n$ over a field of characteristic $p \le n$ is \textbf{not} a highest weight category.

Let $\Lambda$ be a finite poset (finiteness can be relaxed, but for our purposes we won't need to), and let $\mathcal{C}$ be a $k$-linear artinian category with simple objects $L(\lambda)$ indexed by $\Lambda$.  Then a \emph{highest weight structure} on $\mathcal{C}$ is a collection of \emph{standard} objects $\Delta(\lambda)$ for $\lambda \in \Lambda$ such that:

\begin{itemize}
\item $\text{Hom}(\Delta(\lambda),\Delta(\mu)) \ne 0 \implies \lambda \le \mu.$
\item $\text{End}(\Delta(\lambda)) \cong k$ for all $\lambda \in \Lambda$
\item $\mathcal{C}$ has enough projectives, and the projective cover $P(\lambda)$ of $L(\lambda)$ surjects onto $\Delta(\lambda)$ with a kernel which admits a filtration with subquotients of the form $\Delta(\mu)$ for $\mu > \lambda$.
\end{itemize}

Such categories have a rich structure, and we will only be recalling the aspects of the theory which we will need for the purposes of this paper.

\medskip

Suppose $\Lambda' \subset \Lambda$ is a poset ideal in that it is downward closed (i.e. $\lambda \in \Lambda' \implies \mu \in \Lambda'$ for all $\mu < \lambda$).  Let $\mathcal{C}_{\Lambda'}$ denote the Serre subcategory generated by $L(\lambda)$ with $\lambda \in \Lambda'$, then:

\begin{proposition} \label{hwsubcat}
\ 
\begin{enumerate}
\item $\mathcal{C}_{\Lambda'}$ is a highest weight category with respect to the poset $\Lambda'$ and standard objects $\Delta(\lambda)$.

\item If $\lambda$ is maximal inside $\Lambda'$ then $\Delta(\lambda)$ is projective inside $\mathcal{C}_{\Lambda'}$.

\item If $\lambda$ is maximal inside $\Lambda'$ then the there exists an injective hull $\nabla(\lambda)$ of $L(\lambda)$ inside $\mathcal{C}_{\Lambda'}$.  These objects are called costandard, and do not depend on the choice of $\Lambda'$ (provided $\lambda$ is maximal in it).

\end{enumerate}

\end{proposition}

\noindent \textbf{Remark:}  $\Delta(\lambda)$ and $\nabla(\lambda)$ are not in general projective or injective in all of $\mathcal{C}$, only within $\mathcal{C}_{\Lambda'}$.

\medskip

Let $\mathcal{C}^\Delta$ and $\mathcal{C}^\nabla$ denote the full subcategories of objects in $\mathcal{C}$ admitting a filtration by standard and costandard objects respectively.

\begin{proposition} For all $\lambda \in \Lambda$ there is a unique (up to isomorphism) object $T(\lambda)$ satisfying:

\begin{enumerate}
\item $T(\lambda) \in \mathcal{C}^\Delta \cap \mathcal{C}^\nabla$ 
\item The irreducible composition factors of $T(\lambda)$ are of the form $L(\mu)$ with $\mu \le \lambda$, and $L(\lambda)$ appears with multiplicity one.
\item $T(\lambda)$ is indecomposable.

\end{enumerate}

\end{proposition}

These modules $T(\lambda)$ and direct sums thereof are called \emph{tilting} modules.  A tilting module $T$ is said to be a \emph{full} tilting module if it contains $T(\lambda)$ as a summand for all $\lambda \in \Lambda$.  We have the following theorem known as \emph{Ringel duality}:

\begin{theorem}\label{ringel}  Let $T$ be a full tilting module for a highest weight category $\mathcal{C}$, and let $A := \text{End}_{\mathcal{C}}(T)$ be its endomorphism algebra.  Then:

\begin{enumerate}

\item The category $\mathcal{C}^\vee := A\text{-mod}$ of finitely generated $A$ modules is a highest weight category with respect to the poset $\Lambda^{op}$ and standard objects $\text{Hom}(\Delta(\lambda),T)$. It is called the Ringel dual category to $\mathcal{C}$.

\item $\mathcal{C}^\vee$ does not depend on the choice of full tilting module $T$, in other words the algebras $A$ we get in the previous part are Morita equivalent for different $T$.

\item There is natural equivalence $\mathcal{C} \cong (\mathcal{C}^{\vee})^{ \vee}$ respecting the highest weight structure.

\end{enumerate}

\end{theorem}

Of particular importance to us will be the following corollary which says roughly that a highest weight category is completely determined by its subcategory of tilting modules.  

Suppose $\mathcal{C}$ and  $\mathcal{C}'$ are highest weight categories on the same poset $\Lambda$. Let $\text{Tilt}(\mathcal{C})$ and $\text{Tilt}(\mathcal{C}')$  denote the full subcategories of tilting modules with indecomposable tilting modules $T(\lambda)$ and $T'(\lambda)$ respectively for $\lambda \in \Lambda$.  

\begin{corollary}\label{ringelcor}
Any equivalence of categories $\text{Tilt}(\mathcal{C}) \cong \text{Tilt}(\mathcal{C}')$ sending $T(\lambda)$ to $T'(\lambda)$ for all $\lambda \in \Lambda$ extends to an equivalence of categories $\mathcal{C} \cong \mathcal{C}'$ respecting the highest weight structures.
\end{corollary}

\noindent \textbf{Proof:}  Let $T$ and $T'$ denote minimal full tilting modules in $\text{Tilt}(\mathcal{C})$ and $\text{Tilt}(\mathcal{C}')$ respectively, in other words they are a direct sum of all the indecomposable tilting modules with multiplicity one.  The equivalence of categories $\text{Tilt}(\mathcal{C}) \cong \text{Tilt}(\mathcal{C}')$ gives an isomorphism of algebras $\text{End}_{\mathcal{C}}(T) \cong \text{End}_{\mathcal{C}'}(T') =: A$.  Then the previous theorem gives us canonical equivalences between both $\mathcal{C}$ and  $\mathcal{C}'$, and $(A\text{-mod})^\vee$, identifying both $\text{Tilt}(\mathcal{C})$ and $\text{Tilt}(\mathcal{C}')$ with the tilting modules. $\square$

\medskip

\noindent \textbf{Remark:} We could have also developed much of this section at the level of rings instead of categories. We will say that a finite dimensional algebra over a field $k$ is \emph{quasi-hereditary} if its category of finitely generated modules admits the structure of a highest weight category. Such algebras may also be characterized internally in terms of their poset of ideals, but we won't use this characterization.

\end{subsection}

\begin{subsection}{Schur algebras and the Schur functor}\label{Schursect}

Another family of algebras that will be important to us will be the Schur algebras $S(D,n)$. These algebras are well studied and govern the polynomial representations of $GL(V)$, where $V$ is an $D$ dimensional vector space (we will be mostly interested in $D$ large, hence the capital letter). For a background on their structure and representation theory we refer to \cite{DonkinBook}, but for our particular purposes we will be primarily following the exposition in \cite{Hemmer}. They are defined as:

$$S(D,n) \cong \text{End}_{S_n}(V^{\otimes n})$$

In particular these algebras are known to be quasi-hereditary, and hence the category of modules over $S(D,n)$ is a highest weight category with the underlying poset being the dominance order on partitions of size n.  If $\lambda$ is a partition of $n$ with at most $D$ parts we will let $\Delta(\lambda), L(\lambda), P(\lambda)$   and $ T(\lambda)$ denote the Weyl (standard), irreducible, projective, and tilting modules corresponding to $\lambda$.

For us the main fact about Schur algebras we will care about is the existence of the Schur functor $\mathcal{F}: S(D, n)\text{-mod} \rightarrow Rep(S_n)$, along with its right adjoint $\mathcal{G}$ which  satisfy:

\begin{theorem}\textbf{(For a good summaries see \cite{Hemmer} or \cite{HN})}\label{schurfunct}
\begin{enumerate}

\item $\mathcal{F}$ is exact, and $\mathcal{G}$ is left exact.

\item $\mathcal{F}(\Delta(\lambda)) = S_\lambda  \hspace{1cm} \mathcal{F}(L(\lambda)) = D_\lambda \text{ if } \lambda \text{ is } p \text{-restricted} \text{ and } 0 \text { otherwise}$

$\mathcal{F}(P(\lambda)) = Y(\lambda)  \hspace{.65cm}  \mathcal{F}(T(\lambda)) = Y(\lambda')\otimes \text{sgn}$

\item  $\mathcal{F}\circ \mathcal{G} \cong \text{Id}$

\item 
$\mathcal{G}(Y(\lambda)\otimes \text{sgn}) = T(\lambda')$

\item In characteristic at least $5$ with $D \ge n$ then $\mathcal{F}$ and $\mathcal{G}$ restrict to equivalences between the categories of $S(D,n)$-modules with a Weyl filtration, and of $S_n$-modules with a dual Specht filtration. In particular in this case we have that $\mathcal{G}(S_\lambda) = \Delta(\lambda)$.

 \end{enumerate}

\end{theorem}

For a good overview of these functors we'll recommend section $2$ of \cite{Hemmer}, many of these results can also be found in \cite{HN} or \cite{DonkinBook}.  A proof of part $4$ that works in arbitrary characteristic can be found in section $7$ of \cite{DPS}.  We'll note that when we were writing an earlier version of this paper we were only aware of the proof in \cite{HN}, which requires characteristic at least $5$. Now our main results will go through without any restrictions on characteristic.

Really these are a family of functors depending on $D$ and $n$, but the convention seems to be to just always denote them as $\mathcal{F}$ and $\mathcal{G}$ without any subscript or superscript, so we will stick with that.

Note that these are well behaved with respect to the labeling of irreducibles by $p$-restricted partitions, but for our purposes it will be more natural to use the indexing of irreducibles by $p$-regular partitions. So it will be useful for us to correct this by defining the $\emph{twisted Schur functors}$ by:

$$\tilde{\mathcal{F}}(M) = F(M) \otimes \text{sgn} \hspace{1cm} \tilde{\mathcal{G}}(M) = G(M \otimes \text{sgn})$$

We have the following corollary describing some immediate properties of these functors:

\begin{corollary}\label{twistedschur}
\
\begin{enumerate}

\item $\tilde{\mathcal{F}}$ is exact, $\tilde{\mathcal{G}}$ is left exact and right adjoint to $\tilde{\mathcal{F}}$.

\item $\tilde{\mathcal{F}}(\Delta(\lambda)) = S^{\lambda'}  \hspace{2cm} \tilde{\mathcal{F}}(L(\lambda)) = D^{\lambda'} \text{ or } 0$

$\tilde{\mathcal{F}}(P(\lambda)) = Y(\lambda) \otimes \text{sgn}   \hspace{.82cm}  \tilde{\mathcal{F}}(T(\lambda)) = Y(\lambda')$

\item  $\tilde{\mathcal{F}}\circ \tilde{\mathcal{G}} \cong \text{Id}$

\item $\tilde{\mathcal{G}}(Y(\lambda)) = T(\lambda')$

\item In characteristic at least $5$ and $D \ge n$ then $\tilde{\mathcal{F}}$ and $\tilde{\mathcal{G}}$ restrict to equivalences between the categories of $S(D,n)$-modules with a Weyl filtration, and of $S_n$-modules with a Specht filtration. In particular $\tilde{\mathcal{G}}(S^\lambda) = \Delta(\lambda')$

\end{enumerate}
\end{corollary}

\noindent \textbf{Proof:} This is just translating the previous theorem into this language, at a couple of points using the identities $S^\lambda = S_{\lambda'} \otimes \text{sgn}$ and $D^\lambda = D_{\lambda'} \otimes \text{sgn}$. $\square$

\medskip

\noindent \textbf{Remark:} For $n$ larger than the characteristic of $k$,  $k[S_n]$ has infinite global homological dimension, but its representation theory is closely related to that of $S(D,n)$ which has finite global homological dimension.  This motivates the following intuition: ``$\text{Spec}(S(D,n))$ is a non-commutative resolution of singularities for $\text{Spec}(k[S_n])$, and the twisted Schur functors correspond to pushing forward and pulling back sheaves". We won't look into making this statement any more precise or rigorous, but it is some justification for why we will pass to Schur algebras at one point to make things easier.

\end{subsection}

\begin{subsection}{The ring of integer valued polynomials}\label{polysection}

Let $\mathcal{R} \subset \mathbb{Q}[x]$ be the set of all rational polynomials $f(x)$ such that $f(n)$ is an integer for all integers $n$.  We have the following well known theorem:

\begin{proposition}\textbf{(P\'olya 1915 \cite{Polya})}\label{binom}
$\mathcal{R}$ is freely generated as an abelian group by the binomial coefficient polynomials $${x \choose k} = \frac{x(x-1)\dots(x-k+1)}{k!}$$
\end{proposition}

Of particular interest to us will be the reductions of integer valued polynomials modulo a prime $p$. For binomial coefficients we have the celebrated theorem of Lucas:

\begin{proposition}{\textbf{(Lucas 1878 \cite{Lucas})}}\label{lucas}
Let $n = n_0 + n_1p+n_2p^2+ \dots + n_lp^l$ and $m = m_0 + m_1p+m_2p^2+ \dots + m_lp^l$  be the base $p$ expansions for positive integers $n$ and $k$ (possibly with extra zero digits added to make them the same length.) Then:

$$ {n \choose m} \equiv {n_0 \choose m_0} {n_1 \choose m_1} {n_2 \choose m_2} \dots {n_l \choose m_l}  \mod p$$

\end{proposition}

For an integer polynomial $f$ if we let $f_p: \mathbb{Z} \rightarrow \mathbb{F}_p$ be the reduction of $f$ modulo $p$ we have the following lesser known but useful corollary:

\begin{corollary} \label{period}
$f_p(n)$ is periodic in $n$ with period dividing $p^{\lceil log_p(deg(f)) \rceil} $
\end{corollary}

\noindent \textbf{Proof:} By the $\mathbb{Z}$-linearity of reduction modulo $p$ and Prososition \ref{binom} it is enough to check this when $f$ is a binomial coefficient polynomial.  For these, Lucas' theorem tells us that the reduction modulo $p$ of ${n \choose m}$ only depends on the last $\lceil log_p(m) \rceil$ digits of $n$ in base $p$, as desired. $\square$

\medskip

 For any $m \in \mathbb{N}$ the function $t \to {t \choose m}$ extends to a continuous function from $\mathbb{Z}_p$ to itself, which can then be reduced modulo $p$ to obtain a continuous function from $\mathbb{Z}_p$ to $\mathbb{F}_p$. Using Lucas' theorem can see that ${t \choose m}$ modulo $p$ only depends on the last $\lceil log_p(m) \rceil$ base $p$ digits of $t$.  In particular we have evaluation maps $ev_t: \mathcal{R} \rightarrow \mathbb{F}_p$, for every $t \in \mathbb{Z}_p$. 

We include the following proposition to help motivate why taking a $p$-adic limit might be a natural thing to do whenever this ring $\mathcal{R}$ appears.  It says roughly that the closed points of $Spec(\mathcal{R})$ with residue field of characteristic $p$ are indexed by the $p$-adic integers.

\begin{proposition}\label{specr}
Any homomorphism $\phi: \mathcal{R} \rightarrow k$ of rings from $\mathcal{R}$ into a field $k$ of characteristic $p$ factors as $ev_t$ for some $t \in \mathbb{Z}_p$, followed by the inclusion of $\mathbb{F}_p$ into $k$.

\end{proposition}

\noindent \textbf{Proof:} First let's consider maps $\phi$ to $\mathbb{F}_p$. Let $m = m_0 + m_1p + \dots + m_lp^l$ be the base $p$ expansion of a positive integer $m$.  Lucas's theorem tells us that the polynomial

$$ F(x) =  \frac{{x \choose m} - {x \choose m_0} {{x \choose p} \choose m_1} {{x \choose p^2} \choose m_2} \dots {{x \choose p^l} \choose m_l}}{p}$$ 

\noindent is integer valued, and hence $pF(x)$ gets sent to zero under $\phi$.  This implies that the images of ${x \choose m}$ for all $m$ are determined by the images of $x, {x \choose p}, {x \choose p^2}, \dots$  We may then interpret these values as the base $p$ digits of some $t \in \mathbb{Z}_p$ and conclude that $\phi = ev_t$ since they agree on a basis for $\mathcal{R}$.

To see any map into an arbitrary field of characteristic $p$ must factor through this note that for any $f(x) \in \mathcal{R}$ the image under $\phi$ of $f(x)^p - f(x)$ in $k$ is $p$ times the image of $\frac{f(x)^p-f(x)}{p} \in \mathcal{R}$. Therefore the image of $f(x)$ is fixed by the Frobenius map and hence is in $\mathbb{F}_p$.  $\square$

\end{subsection}

\begin{subsection}{The stable order on partitions}

Here's an observation that everyone working in representation stability implicitly seems to know, but as far as we can tell does not explicitly appear written down anywhere:

\begin{proposition}
If $n, m > 2r$ the correspondence $\lambda(n) \rightarrow \lambda(m)$ for $|\lambda| \le r$ respects the dominance order and is an equivalence of partially ordered sets.
\end{proposition}

\noindent \textbf{Proof:} Obvious. $\square$

\medskip

The point of stating this explicitly is that it allows us to define the \emph{stable order} on the set of all partitions by saying that $\lambda \preceq \mu$ iff $\lambda(n) \ge \mu(n)$ in the dominance order for $n \gg 0$.

Note that there is a change of direction in the inequalities. Of course this is just a convention, but for many purposes the reverse dominance order is the more natural one when working with symmetric groups. Here are a few reasons why we go with this convention: 

\begin{itemize}
\item The cellular structures on the group algebras of symmetric groups are defined with respect to the reverse dominance order.

\item It is compatible with the inductive description of the representation theory of symmetric groups in section \ref{modreview}.

\item The stable order is an extension of the containment ordering on partitions, and in particular the empty partition is the unique minimal element.

\end{itemize}

 Many of the theorems we state will involve looking at the set of partitions $\lambda$ of size at most $r$.  This is convenient for stating bounds on how large various parameters need to be, but the only property about the set that we need for many of these theorems to go through is that it is a poset ideal in the stable order.

\end{subsection}
\end{section}

\begin{section}{Stability and periodicity}

This section will contain the main results of this paper.   We first restrict our attention to permutation modules where everything can be understood combinatorially, then we extend our to more general representations using Schur-Weyl duality and Ringel duality.  

As some motivation for why we proceed this way we note that permutation modules $M(\lambda)$ and direct sums thereof are a particularly nice class of representations. In particular they are:

\begin{enumerate}
\item Defined over $\mathbb{Z}$.

\item Self dual.

\item Appropriately preserved under induction and restriction to Young subgroups.

\item Closed under taking tensor products.
\end{enumerate}

Moreover, Theorem \ref{modreview} suggests that these permutation modules ``see" a lot of the representation theory of the symmetric groups in characteristic $p$.  So we will begin by understanding stability and periodicity for these modules.

\begin{subsection}{Stability and polynomiality over $\mathbb{Z}$}

For a partition $\mu = (\mu_1, \mu_2, \dots, \mu_\ell)$ of $n$ we may realize $M(\mu)$ as the linearization of the natural action of $S_n$ on the collection of set-partitions $[n] = A_1 \cup A_2 \cup \dots \cup A_\ell$ with $|A_i|=\mu_i$.

A tabloid of shape $\mu$ and type $\lambda$ is a labeling of the boxes of a Young diagram of shape $\mu$ with $\lambda_1$ boxes labeled $1$,  $\lambda_2$ boxes labeled $2$, etc. considered up to equivalence of permuting the entries in each row.   For a tabloid $\tau$ we will let $\tau_{i,j}$ denote the number of boxes in row $i$ labeled by $j$.

Tabloids of shape $\mu$ and type $\lambda$ index the double cosets of $S_n$ by the Young subgroups of shape $\mu$ and $\lambda$. They can be thought of as keeping track of the relative positions a pair of set-partitions of an $n$ element set, one of shape $\mu$ and one of shape $\lambda$. There is an obvious bijection between tabloids and row-semistandard tableaux of the same shape and type, and some authors prefer that convention.  We will stick to working with tabloids.

For a tabloid $\tau$ of shape $\mu$ and type $\lambda$ we may define a homomorphism $h^\tau: M(\mu) \rightarrow M(\lambda)$ by the formula:

$$h^\tau(\langle A_1 \cup A_2 \cup \dots \cup A_\ell \rangle) = \sum_{|A_i \cap B_j |= \tau_{i,j}} \langle B_1 \cup B_2 \cup \dots \cup B_{\ell'} \rangle$$

\noindent which we will refer to as a \emph{Carter-Lusztig basis map}, as they were (essentially) defined in Carter and Lusztig's seminal paper \cite{CL} on modular representation theory.  In words, this map just sends a set-partition $A$ of shape $\mu$ to a formal sum of all set-partitions $B$ of shape $\lambda$ which are in the relative position to $A$ indexed by $\tau$. We have the following theorem:

\begin{theorem} \textbf{(James \cite{James} Theorem 13.19, after \cite{CL})} 
Let $R$ be an arbitrary commutative ring. $\text{Hom}_{{R}[S_n]}(M(\mu), M(\lambda))$ is a free $R$ module generated by the maps $h^\tau$.
 
 
\end{theorem}

\noindent \textbf{Remark:}  Our description of the modules $M(\lambda)$ and the maps $h^\tau$ is somewhat different from the descriptions in \cite{CL} or \cite{James}, but nevertheless can be seen to be equivalent. We feel this description makes some of the combinatorics more transparent.

\medskip

Now we want to pass to the stable setting, so let us consider tabloids of shape $\mu(n)$ and type $\lambda(n)$.  Such a Young tabloid is completely determined by its restriction to the rows after the first one.  Moreover, the collection of possible restrictions is independent of $n$ for $n > 2\text{max}\{|\mu|,|\lambda|\}$. We will call this the collection of \emph{stable tabloids} of shape $\mu$ and type $\lambda$.

We should note that stable tabloids of shape $\mu$ and type $\lambda$ are not actually a subset of tabloids of shape $\mu$ and type $\lambda$.  Rather, they can be described combinatorially as a labeling of the boxes of a Young diagram of shape $\mu$ with at most $\lambda_1$ boxes labeled $2$,  at most $\lambda_2$ boxes labeled $3$, etc., and an arbitrary number of boxes labeled $1$,  again considered only up to permuting the rows.

\medskip
This allows us to identify the corresponding hom-spaces for different values of $n$, provided $n$ is large enough. Let $\overline{ \text{Hom}}(M(\mu), M(\lambda))$ be the stable limit of the spaces $\text{Hom}(M(\mu(n)), M(\lambda(n)))$ where we identify maps in different $n$ corresponding to tabloids that agree outside of the first rows. We will refer to this as a \emph{stable hom-space}, as an abelian group it is just the free abelian group generated by the stable tabloids of shape $\mu$ and type $\lambda$.

If $\tau$ is a stable tabloid of $\mu$ obtained by ignoring the first row of a tabloid $\tau(n)$ of $\mu(n)$ for $n\gg0$ we'll let $f^\tau$ denote the corresponding element of $\overline{ \text{Hom}}(M(\mu), M(\lambda))$.  For $f$ in one of these stable hom-spaces we let $f_n$ denote the corresponding map for the symmetric group $S_n$ (in particular $f^{\tau}_n = h^{\tau(n)}$).
  
\medskip

\noindent \textbf{Remark:} We may similarly define stable hom-spaces from Specht modules to permutation modules (provided we aren't in characteristic two), as well as stable hom-spaces between finite direct sums of these types of permutation modules. 

\medskip

Of course we don't just care about the hom-spaces between permutation modules, we also need to understand the composition maps. For the following two results we will state them for maps between permutation modules, but they also hold (with essentially the same proof) for finite direct sums of permutation modules as well as if one of the maps is from a Specht module into a permutation module.
 
\begin{proposition}\label{polycomp} For stable maps $f^\alpha, f^\beta$ in the stable Carter-Lusztig bases of $\overline{ \text{Hom}}(M(\mu), M(\lambda))$ and $\overline{\text{Hom}}(M(\lambda), M(\nu))$ respectively, for $n \gg 0$ we have that:

$$f^\beta_n \circ f^\alpha_n = \sum_\tau p_\tau(n) f^\tau_n$$

\noindent where the $p_\tau(n)$ are integer valued polynomials of degree at most $\text{max}\{|\mu|,|\lambda|, |\nu|\}$, and $\tau$ runs over the stable tabloids of shape $\mu$ and type $\nu$.

\end{proposition}

\begin{proof}
Computations of compositions in the (ordinary, non-stable) Carter-Lusztig bases can be computed explicitly and combinatorially in terms of the corresponding tabloids (For example, see \cite{FM} for explicit examples of such calculations). In particular the coefficients are finite sums and products of binomial coefficients in the variables $\tau_{i,j}$ for the different tabloids $\tau$ involved.  For the families of maps we look at these variables are all either a constant independent of $n$ or of the form $n-r$ for some fixed $r$. The result follows immediately. $\square$
\end{proof}

We will also be interested in restrictions of permutation modules to Young subgroups, tensor products of permutation modules, and taking duals. We have the following:

\begin{proposition}\label{indep}
The following are independent of $n$ for $n \gg 0$:

\begin{enumerate}

\item The decomposition of the restriction of $M(\lambda(n))$ to a Young subgroup $S_\ell \times S_{n-\ell}$ into modules of the form $M(\mu)\otimes M(\nu(n-\ell))$ for fixed $\ell$, as well as an analogous statement for induction the other direction.

\item In the setting of the previous part, the decomposition of a Carter-Lusztig basis map $h^{\tau(n)}$ into a direct sum of products $h^{\alpha} \otimes h^{\beta(n)}$, and an analogous statement for induction.

\item The decomposition of a tensor product $M(\lambda(n)) \otimes M(\mu(n))$ into permutation modules $M(\nu(n))$.

\item The decomposition of products of Carter-Lusztig basis maps $h^{\tau(n)} \otimes h^{\alpha(n)}$ as a direct sum of maps $h^{\beta(n)}$.

\item The map on stable hom-spaces $\overline{ \text{Hom}}(M(\mu), M(\lambda)) \rightarrow \overline{ \text{Hom}}(M(\lambda), M(\mu))$ given by identifying them with the corresponding hom-spaces for $M(\mu(n))$ and $M(\lambda(n))$ and applying duality.

\end{enumerate}

\end{proposition}

 These restrictions, tensor products, and dualities can all be computed explicitly in terms of the combinatorial descriptions of permutation modules and the Carter-Lusztig basis maps we gave at the beginning of this section. One then just needs to verify that the combinatorics involved stabilizes in $n$. Since all of these proofs are of a similar flavor (and no different from the characteristic zero case), we will give a proof for part 1 and omit the rest.

\medskip

\noindent \textbf{Proof of (1):} The restriction of a permutation module $M(\lambda)$ to a Young subgroup $S_\ell \times S_{n-\ell}$ can be computed by Mackey's decomposition theorem.  This amounts to summing over intersections of the subgroups $g(S_\ell \times S_{n-\ell})g^{-1}$ with the Young subgroup corresponding to $\lambda$ as $g$ runs over a choice of representatives for the double cosets. Unraveling this we get:

$$Res(M(\lambda)) = \bigoplus_{\substack{{\tilde{\mu}+\tilde{\nu} = \lambda} \\ {|\mu|=\ell}}} M(\mu)\otimes M(\nu)$$

Where $\tilde{\mu}$ and $\tilde{\nu}$ run over compositions rather than partitions, meaning they are allowed to have empty rows and not be ordered by size with addition defined by $(\tilde{\mu}_1, \tilde{\mu}_2, \dots, \tilde{\mu}_\ell) + (\tilde{\nu}_1, \tilde{\nu}_2, \dots, \tilde{\nu}_\ell) = (\tilde{\mu}_1+\tilde{\nu}_1, \tilde{\mu}_2+\tilde{\nu}_2, \dots, \tilde{\mu}_\ell+\tilde{\nu}_\ell)$, and $\mu, \nu$ denote the partitions we get by ignoring empty rows and rearranging them in decreasing order. 

If we look at this decomposition of $Res(M(\lambda(n)))$ for $\ell$ fixed and $n$ growing, the possibilities for $\tilde{\mu}$ stabilize as soon as  $n -  |\lambda| > \ell$ and the complementary compositions $\tilde{\nu}(n)$ always get reordered into the same (up to the first row) partitions $\nu(n)$ as soon as $n -  |\lambda| > \ell + |\lambda|$.

For induction the story is even simpler as permutation modules are just the trivial module induced from a Young subgroup. It follows that:

$$\text{Ind}(M(\mu)\boxtimes M(\lambda)) = M(\mu \cup \lambda)$$

Where $\mu \cup \lambda$ denotes the partition we get by shuffling together the parts of $\mu$  with the parts of $\lambda$ and reordering them in decreasing order. In particular this stabilizes as $\mu \cup \lambda(n) = (\mu \cup \lambda)(n)$ for $n > 2|\lambda| + |\mu|$.  $\square$

\medskip

\noindent \textbf{Remark:} In some sense, for large $n$ the only dependence on $n$ for permutation modules of the form $M(\lambda(n))$ is in the composition law.  Those familiar with Deligne's categories $\underline{\text{Rep}}(S_t)$, in particular the decription by Comes and Ostrik in \cite{CO}, will find this extremely reminiscent of what happens there. See section \ref{prank} for an explicit connection.

\end{subsection}

\begin{subsection}{Periodicity in positive characteristic}

Now let's see what this tells us about modular representations by tensoring everything with an algebraically closed field $k$ of characteristic $p$.

\begin{proposition}
\label{comp}

For $n \gg 0$ the $k$-linear composition map$$\circ_n : \overline{ \text{Hom}}(M(\mu), M(\lambda)) \otimes \overline{\text{Hom}}(M(\lambda), M(\nu))  \rightarrow \overline{ \text{Hom}}(M(\mu), M(\lambda))$$

\noindent obtained by identifying the stable hom-spaces with the genuine hom-spaces for $S_n$ only depends on $n$ modulo $p^m$ for $m = \lceil log_p(\text{max}\{|\mu|, |\lambda|, |\nu|\}) \rceil$.

\end{proposition}

\noindent \textbf{Proof:} Extend Prososition \ref{polycomp} linearly to the entire hom-spaces, and then apply Corollary \ref{period}. $\square$

\medskip

Let $Perm^{\le r}_n \subset Rep(S_n)$ denote the full subcategory of representations where objects are direct sums of permutation modules $M(\lambda(n))$ with $|\lambda| \le r$.  We are now ready to state our first periodicity result for the modular representations of symmetric groups:

\begin{theorem}\label{permequiv}
For all $n, m > 2r$ satisfying $p^{\lceil log_p(r) \rceil } | (n-m)$ we construct an equivalence of categories $Perm^{\le r}_n \cong Perm^{\le r}_m$.

\end{theorem}

\noindent \textbf{Proof:}  The equivalence is given by sending $M(\lambda(n))$ to $M(\lambda(m))$ and by sending $f_n$ to $f_m$ for stable maps $f$.  The previous proposition ensures that this actually defines a functor. $\square$

\medskip

If we let $Young^{\le r}_n$ be the full subcategory of $Rep(S_n)$ with objects being direct sums of Young modules $Y(\lambda(n))$ with $|\lambda| \le r$, then we get the following small but useful improvement on the previous theorem.

\begin{corollary}\label{youngequiv}
For all $n, m > 2r$ satisfying $p^{\lceil log_p(r) \rceil } | (n-m)$ there is an equivalence of categories $Young^{\le r}_n \cong Young^{\le r}_m$.
\end{corollary}

\noindent \textbf{Proof:} This is immediate since $Young^{\le r}_n$ is equivalent to the Karoubian envelope of $Perm^{\le r}_n$, by Theorem \ref{young} $\square$

\medskip

As mentioned before, the categories of permutation modules (and similarly Young modules) are also well behaved with respect to induction and restriction to Young subgroups as well as taking tensor products. As one would hope, the equivalences of categories we just defined are compatible with these additional structures in the following sense:

\begin{proposition}\label{permfunctors}
For fixed $\ell$ and appropriate $r$, $n$, and $m$, the following diagrams commute in the sense that there are natural isomorphisms between the two functors obtained by going from the top left to the bottom right.

\begin{enumerate}[(1)]

\item \[
\begin{tikzcd}
Perm^{\le r}_n  \arrow{r}{Res} \arrow[swap]{d}{} & Perm_\ell \boxtimes Perm^{\le r}_{n-\ell} \arrow{d}{} \\
Perm^{\le r}_m  \arrow{r}{Res} & Perm_\ell \boxtimes Perm^{\le r}_{m-\ell}
\end{tikzcd}
\]


\item
\[
\begin{tikzcd}
Perm_\ell \boxtimes Perm^{\le r}_n \arrow{r}{Ind} \arrow[swap]{d}{} & Perm^{\le r+\ell}_{n+\ell} \arrow{d}{} \\
Perm_\ell \boxtimes Perm^{\le r}_m  \arrow{r}{Ind} & Perm^{\le r+\ell}_{m+\ell} 
\end{tikzcd}
\]

\item \[
\begin{tikzcd}
Perm^{\le r}_n\boxtimes Perm^{\le r}_n \arrow{r}{\otimes} \arrow[swap]{d}{} & Perm^{\le 2r}_n \arrow{d}{} \\
Perm^{\le r}_m \boxtimes Perm^{\le r}_m \arrow{r}{\otimes} & Perm^{\le 2r}_m
\end{tikzcd}
\]

\end{enumerate}
Where all the vertical maps are the equivalences of categories from Theorem \ref{permequiv}, possibly tensored with each other or with the identity functor on $Perm_\ell$.

\end{proposition}

\noindent \textbf{Proof:} These all follow immediately from Theorem \ref{indep} reduced modulo $p$. $\square$

\end{subsection}

\begin{subsection}{Beyond permutation modules}

The goal now is to extend the equivalences of categories in Theorem \ref{permequiv} to an equivalence of categories between certain abelian subcategories of representations of the symmetric groups $S_n$ and $S_m$.   For this it will be convenient to pass through Schur-Weyl duality and first prove an analogous result for Schur algebras.

 Let $\mathcal{S}(D,n)^{\le r}$ denote the Serre subcategory of $S(D,n)\text{-mod}$ generated by $L(\lambda(n)')$ with $|\lambda| \le r$. We have the following theorem:

\begin{theorem}\label{schurequiv}
Under the conditions of Theorem \ref{permequiv} for $n, m,$ and $r$ there is an equivalence of categories between $\mathcal{S}(D,n)^{\le r}$ and $\mathcal{S}(D',m)^{\le r}$ for $D \ge n$ and $D' \ge m$.
\end{theorem}

\noindent \textbf{Proof:} First we note that Prososition \ref{hwsubcat} tells us that $\mathcal{S}(D,n)^{\le r}$ is a highest weight category, and moreover the characterization of tilting modules tells us its tilting modules are just the tilting modules $T(\lambda(n)')$ for $|\lambda| \le r$ from the larger category.

Then Corollary \ref{twistedschur} tells us that the twisted Schur functors restrict to an equivalence of categories between $\text{Tilt}(\mathcal{S}(D,n)^{\le r})$ and $Young^{\le r}_n$.  Doing the same for $D'$ and $m$ and applying Corollary \ref{youngequiv} gives an equivalence of categories between $\text{Tilt}(\mathcal{S}(D,n)^{\le r})$ and $\text{Tilt}(\mathcal{S}(D',m)^{\le r})$ sending $T(\lambda(n)')$ to $T(\lambda(m)')$ for all $|\lambda| \le r$.

This puts us exactly in the situation of Corollary \ref{ringelcor}, and applying this gives the desired result. $\square$ 

\medskip

\noindent \textbf{Remark:}  In the proof of Corollary \ref{ringelcor} we took a minimal full tilting module, which in this setting would correspond to a direct sum of all Young modules.  One downside for us is that Young modules are highly non-uniform in $n$ and difficult to do actual computations with.  If one wanted to do actual computations with these endomorphism algebras we'd recommend instead using a non-minimal full tilting module corresponding to a sum of appropriate permutation modules, this is better behaved but still Morita equivalent to the one for a minimal tilting module.

\medskip

Now we'd like to again apply the twisted Schur functors to get an analogous result for symmetric groups, however some care needs to be taken in stating the result. The complication comes from the fact that the Schur functor $\mathcal{F}$ kills irreducible objects $L(\lambda)$ for $p$-singular $\lambda$. So it is not necessarily true that the essential image of $\mathcal{S}(D,n)^{\le r}$ is the Serre subcategory generated by $Young_n^{\le r}$, as there could be extra extensions corresponding to objects in $S(D,n)\text{-mod}$ with a composition factor $L(\lambda)$ for some $p$-singular $\lambda$ outside our poset ideal.

Nevertheless we can still get an interesting equivalence of abelian subcategories this way. Let $Rep(S_n)_{im}^{\le r}$ denote the essential image of $\mathcal{S}(D,n)^{\le r}$ under the twisted Schur functor $\tilde{\mathcal{F}}$.  Then indeed we have:

\begin{theorem}\label{symequiv}
Under the conditions of Theorem \ref{schurequiv}, the equivalence of categories from Theorem \ref{permequiv} extends to an equivalence of categories between $Rep(S_n)_{im}^{\le r}$ and $Rep(S_m)_{im}^{\le r}$.
\end{theorem}

\noindent \textbf{Proof:} First we claim that $Rep(S_n)_{im}^{\le r}$ is equivalent to the Serre quotient of $\mathcal{S}(D,n)^{\le r}$ by the Serre subcategory $\mathcal{S}(D,n)^{\le r}_{sing}$ generated by those $L(\lambda(n)')$ with $p$-singular $\lambda$.

For this we will need a categorical result due to Gabriel. Say that an exact functor $\mathcal{D}: \mathcal{A} \rightarrow \mathcal{B}$ admits a \emph{section} if it admits a right adjoint $\mathcal{D}'$ such that the counit of adjunction $\epsilon: \mathcal{D} \circ \mathcal{D}' \rightarrow \text{Id}_{\mathcal{B}}$ is an isomorphism.  

\begin{lemma}\textbf{(Gabriel 1962 \cite{Gabriel} Prop III.2.5)}
Suppose $\mathcal{D}: \mathcal{A} \rightarrow \mathcal{B}$ is an exact functor admitting a section, then $\mathcal{A} / \text{ker}(\mathcal{D}) \cong \mathcal{B}$,  where  $\mathcal{A} / \text{ker}(\mathcal{D})$ denotes the Serre quotient of $\mathcal{A}$ by the Serre subcategory of objects which get sent to zero under $\mathcal{D}$.
\end{lemma}

Our description (Corollary \ref{twistedschur}) of twisted Schur functors, tells us that $\tilde{\mathcal{G}}$ is a section of $\tilde{\mathcal{F}}$, and that the irreducibles in the kernel of $\tilde{\mathcal{F}}$ are exactly those $L(\lambda(n)')$ with $p$-singular $\lambda$. So the lemma applies and gives us the claim.

Then we note that the same holds for $D'$ and $m$ and that the equivalence of categories from Theorem \ref{schurequiv} sends $\mathcal{S}(D,n)^{\le r}_{sing}$ to $\mathcal{S}(D',m)^{\le r}_{sing}$, the theorem follows immediately.  $\square$

\medskip

This theorem has the slightly dissatisfying property that the two subcategories of $Rep(S_n)$ and  $Rep(S_m)$  we are comparing are described extrinsically in terms of Schur functors. However we can easily fix this by shrinking to two subcategories which are smaller but easier to describe intrinsically.  Let $Rep(S_n)^{\le r}$ denote the abelian subcategory of $Rep(S_n)_{im}^{\le r}$ generated by $Perm^{\le r}_n$. It follows trivially that:

\begin{corollary}\label{main}
Under the conditions of the previous theorem, the equivalence of categories between $Perm^{\le r}_n$ and $Perm^{\le r}_m$ of Theorem \ref{permequiv} extends to an equivalence between $Rep(S_n)^{\le r}$ and $Rep(S_m)^{\le r}$.
\end{corollary}

Again there is more we can ask about the representation theory of symmetric groups than just the structure of the abelian categories of representations. In particular one might also about induction, restriction, tensor products, and duality functors on these categories and how they interact with our equivalences of categories. We have:

\begin{proposition}\label{extrastruct}
The functors defining induction, restriction, tensor products, and duality all commute (in appropriate senses, similar to Prososition \ref{permfunctors}) with the equivalences of categories in Corollary \ref{main}.
\end{proposition}

\noindent \textbf{Proof:} We already saw versions of this for permutation modules in Prososition \ref{permfunctors}. All of these functors are exact so this automatically extends to the abelian categories they generate. $\square$

\medskip

\noindent \textbf{Remark:} It is likely that one can prove similar statements for the equivalences of categories in Theorem \ref{symequiv}, but that may require more careful analysis of what these functors correspond to under Schur-Weyl duality.  We haven't pursued this as the categories $Rep(S_n)^{\le r}$ are enough for most of our purposes. Moreover it is unclear to what extent, if any, $Rep(S_n)^{\le r}$ and $Rep(S_n)_{im}^{\le r}$ are actually different.

\end{subsection}

\end{section}

\begin{section}{Applications, corollaries, and conjectures}

\begin{subsection}{Numerical invariants}
Having an equivalence of categories is of course much stronger than equating numerical invariants. Nevertheless we'd like to highlight a few important numeric corollaries to our main theorems, particularly those that are either about well studied invariants in the context of modular representations of symmetric groups, or those which are analogous to known stabilization results in characteristic zero. Fix a prime number $p$.

\begin{proposition} \label{numeric}The following sequences of integers are eventually periodic in $n$ with period a power of $p$:

\begin{enumerate}
\item The $p$-Kostka numbers $(M(\lambda(n)) : Y(\mu(n)))$ recording the multiplicity of $Y(\mu(n))$ as a summand in $M(\lambda(n))$, as well as the dimension of the $\lambda$ weight space in the irreducible $GL_D$ representation $L(\mu)$ (see \cite{Donkin}).

\item The decomposition numbers $[S^{\lambda(n)} : D^{\mu(n)}]$ recording the composition multiplicity of irreducible modules in Specht modules.

\item The modular Littlewood-Richardson coefficients $c_{\nu,\lambda(n)}^{\mu(n + |\nu|)}$ recording the composition multiplicity of $D(\mu(n + |\nu|))$ in $Ind(D^\nu \boxtimes D^{\lambda(n)})$. 

\item The p-Kronecker coefficients $g_{\nu(n),\lambda(n)}^{\mu(n)}$ recording the composition multiplicity of $D^{\mu(n)}$ in $D^{\nu(n)} \otimes D^{\lambda(n)}$.

\end{enumerate}

\end{proposition}

\noindent \textbf{Proof:} Our equivalences of categories descend to isomorphisms of Grothendieck groups, where all of these numerical invariants are defined. The periodicity comes from the fact that our equivalences only exist when $(n-m)$ is divisible by a sufficiently large power of $p$.  $\square$

\medskip

In characteristic zero there is a well known polynomiality result about the characters $\chi_{S^{\lambda(n)}}$ of irreducible $S_n$ modules (which in that case are Specht modules). We have the following modular analog for the irreducible Brauer characters $\hat{\chi}_{D^{\lambda}}$ with `quasi-polynomial' replacing `polynomial':

\begin{proposition}
For any partition $\lambda$ there exists an integer $\ell$ along with integer valued polynomial $P_{\lambda,j} \in \mathbb{Q}[x_1,x_2,x_3,...]$ of degree at most $|\lambda|$ for $j =1, 2, \dots, p^\ell$ such that: $$\hat{\chi}_{D^{\lambda(n)}}(\sigma) =  P_{\lambda,j}(X_1(\sigma), X_2(\sigma), ...)$$

\noindent for all $n \gg 0$ satisfying $n\equiv j \mod p^\ell$, where $X_i(\sigma)$ denotes the number of length $i$ cycles in $\sigma$. 

\end{proposition}
\noindent \textbf{Proof:} Specht modules are defined over $\mathbb{Z}$ and their Brauer characters agree with the usual characters in characteristic zero.  The result then follows inductively by the lower triangularity of the decomposition matrices using the analogous fact in characteristic zero and part $2$ of the previous proposition. $\square$

\end{subsection}

\begin{subsection}{An application to $FI$-modules}
We'd now like to connect the theory developed here to the theory of finitely generated $FI$-modules, so that by extension it is connected to the many applications of that theory outside of representation theory.    

As usual for an $FI$-module $V$ let $V_n$ denote the representation of $S_n$ corresponding to the image of a set of size $n$ under the functor $V$. We have the following proposition:

\begin{proposition}\label{fiprop}
If $V$ is a finitely generated $FI$-module, then there exists an $\ell$ such that for all $n,m\gg0$ with $p^\ell | (n-m)$  the image of $V_n$ under our equivalence of categories from Corollary \ref{main} is isomorphic to $V_m$.
\end{proposition}

\noindent \textbf{Proof:}  This is clearly true for free $FI$-modules, which correspond in our notation to the sequence of representations of the form $V_n = M(\lambda(n))$ for $\lambda = (1^d)$, and direct sums thereof.  Moreover $FI$-maps between free $FI$-modules are easy to describe (see definition 2.2 in \cite{CEFN}) and can be identified with a subspace of one of our stable hom-spaces (one spanned by certain stable Carter-Lusztig maps, in fact).  To get an arbitrary finitely generated $FI$-module we just note that every $FI$-module is a cokernel of a map between free $FI$-modules by \cite{CEFN} proposition 2.3.  $\square$

\medskip

In particular, if $V$ is a finitely generated $FI$-module this gives us the eventual periodicity (with period a power of p) of a number of representation theoretic questions including:

\begin{itemize}
\item The decomposition of $V_n$ into indecomposable summands.

\item The length of the socle filtration of $V_n$, and the irreducible summands of its quotients. In particular this gives us the composition length of $V_n$ as well as the composition multiplicities of the irreducibles $D^{\lambda(n)}$.

\item The dimensions of the cohomology groups $H^i(S_n,V_n)$, recovering a result of Nagpal \cite{Nagpal}.
\end{itemize}

This gives strengthenings of the topological and algebraic applications of $FI$-modules from \cite{CEFN}. In particular the above results hold for sequences $V_n$ of representations such as homology spaces $\mathcal{H}_m(\Gamma_n(\mathfrak{p}),k)$ of congruence subgroups, cohomology spaces $H^m(Conf_n(M), k)$ of configuration spaces, and graded parts $R^{(r)}_J(n)$ of diagonal coinvariant algebras. 

\medskip
\noindent \textbf{Remark:}  In principle one could go through everything more carefully in order to bound the eventual periodicity as well as how fast it reaches the stable period. These should probably just be relatively simple formulas in terms of the degree in which the $FI$-module is generated, and the support of the $FI$-module homology. Since $FI$-modules aren't the focus of our paper we won't pursue this direction any further.

\end{subsection}

\begin{subsection}{Deligne's ultrafilter construction}\label{Deligne}

A \emph{symmetric tensor category} over a field $k$ is a $k$-linear abelian category $\mathcal{C}$ endowed with a biexact $k$-linear monoidal functor $\otimes: \mathcal{C}\times\mathcal{C} \rightarrow \mathcal{C}$ with associativity and commutativity isomorphisms and a unit object $\textbf{1}$ satisfying certain axioms (see \cite{Ostrik}).  A symmetric tensor category $\mathcal{C}$ is called a \emph{pre-Tannakian category} if it additionally satisfies:

\begin{enumerate}
\item $\mathcal{C}$ is essentially small, has finite dimensional hom-spaces, and each object has finite composition length.

\item The natural morphism $k \rightarrow \text{End}(\textbf{1})$ is an isomorphism.

\item Every object $X$ in $\mathcal{C}$ admits a dual $X^*$. In other words, $\mathcal{C}$ is rigid (see \cite{EGNO} chapter $2.10$).

\end{enumerate}

A symmetric tensor functor from a pre-Tannakian category to the category of (super) vector spaces is called a \emph{(super) fiber functor}. A pre-Tannakian category equipped with (super) fiber functor is called a \emph{(super) Tannakian category}.  Tannakian formalism tells us that any (super) Tannakian category is equivalent to a category $Rep(G,\varepsilon)$ of representations of an algebraic (super) group $G$ for which a certain central element $\varepsilon$  satisfying $\varepsilon^2 =1$ acts by the parity endomorphism, endowed with the forgetful functor to (super) vector spaces. (See \cite{Deligne2} for the classical case, and \cite{Deligne3} for the super case.)

A pre-Tannakian category $\mathcal{C}$ is said to be of \emph{subexponential growth} if for any object $X$ the length of $X^{\otimes n}$ is bounded by a function $C_X^n$ for some real number $C_X$ (an object failing this condition is said to be of \emph{superexponential growth}). Being of subexponential growth is a necessary condition for the existence of a super fiber functor.  In characteristic zero Deligne showed (in \cite{Deligne2}) that this is also a sufficient condition. Moreover there exist examples of pre-Tannakian categories in characteristic zero which are not of subexponential growth, and hence do not admit a super fiber functor.

In positive characteristic this condition is known not to be sufficient, there are examples of pre-Tannakian categories of subexponential growth which do not admit a super fiber functor.  Recently Ostrik proposed a relaxation of the notion of super fiber functor in positive characteristic, and conjectured that subexponential growth is a sufficient condition to imply the existence of such a functor (see \cite{Ostrik}).  He also remarked that there are no known pre-Tannakian categories in positive characteristic which are not of subexponential growth. Deligne responded to this remark with a letter containing the following theorem:

\begin{theorem}\textbf{(Deligne 2015)} There exist pre-Tannakian categories in arbitrary characteristic with objects of super-exponential growth.

\end{theorem}

\noindent \textbf{Proof (summarizing Deligne):}  Fix a finite field $k$ and a nontrivial ultrafilter $\mathcal{U}$ on the set of natural numbers and let $Rep(S_\mathcal{U})$ denote the ultraproduct of the categories $Rep_k(S_n)$ for $n$ running over the natural numbers. More explicitly, consider the category whose objects are sequences $\{X_n\}$ where $X_n$ is a representation of $S_n$ over $k$ for each $n \in \mathbb{N}$, and morphisms from $\{X_n\}$ to $\{Y_n\}$ are given by sequences of morphisms $\{\phi_n: X_n \rightarrow Y_n\}$. $Rep(S_\mathcal{U})$ is obtained from this category by quotienting by the relation that two sequences (of objects or morphisms) are the equivalent if they agree on a subset of $\mathbb{N}$ belonging to $\mathcal{U}$.

We may equip $Rep(S_\mathcal{U})$ with a tensor structure given by $\{X_n\} \otimes \{Y_n\} = \{X_n\otimes Y_n\}$.  This makes it into a symmetric tensor category with duals ($\{X_n\}^* = \{X_n^*\}$) and a unit object $\textbf{1} = \{\textbf{1}_n\}$ satisfying $\text{End}(\textbf{1}) \cong k$ (for this it is important that we took $k$ to be a finite field).  However it is not pre-Tannakian since it has infinite dimensional hom-spaces and objects of infinite length. For example we could take the object corresponding to the sequence of representations where $X_n$ is an $n$-dimensional vector space with trivial action of $S_n$, this has infinite length and an infinite dimensional endomorphism algebra.

We can trim this down to a more manageable category $Rep(S_\mathcal{U})_0$ by considering the abelian tensor subcategory generated by the object $X = \{X_n\}$ where $X_n$ denotes the standard $n$ dimensional representation of $S_n$.  James has shown previously that the length of $X_n^{\otimes m}$ is bounded by a function $f(m)$ not depending on $n$ (or the characteristic), in particular this implies that objects in $Rep(S_\mathcal{U})_0$  all have finite length, and similarly one can get that all hom-spaces are finite dimensional.  In particular we get that $Rep(S_\mathcal{U})_0$ is pre-Tannakian.

The length of $X^{\otimes m}$ in $Rep(S_\mathcal{U})_0$ is at least as big as that of the corresponding object in characteristic zero which we know grows superexponentially in $m$, so therefore $X$ is an object of superexponential growth, as desired. $\square$

\medskip

The main results of this paper give us more refined control than James's result on the length of $X_n^{\otimes m}$, so now we will see what we get when we use them in its place in the above analysis. The $p$-adic integers are compact, so every sequence of integers has a unique ultralimit with respect to the ultrafilter $\mathcal{U}$ inside $\mathbb{Z}_p$.  In particular the sequence $a_n = n$ has a limit value $t = t(\mathcal{U}) \in \mathbb{Z}_p$.  We have the following theorem conjectured by Deligne:

\begin{theorem}\label{Delconj}
For two nontrivial ultrafilters  $ \ \mathcal{U}$ and $ \ \mathcal{U}'$, the categories $Rep(S_\mathcal{U})_0$ and $Rep(S_\mathcal{U'})_0$ are equivalent as pre-Tannakian categories iff $t(\mathcal{U}) = t(\mathcal{U}') $.
\end{theorem}

\noindent \textbf{Proof:}  The `only if' direction is easy, one can read off the base $p$ digits of $t$ from the dimensions of exterior powers of $X$.  The other direction, Deligne noted, is equivalent to the abelian subcategory of $Rep(S_n)$ generated by $X_n^{\otimes{m}}$ (along with the data of the functor $\otimes$ restricted to these subcategories)  being independent of $n$ for $n\gg 0$ in a fixed residue class modulo $p^{g(m)}$ for some function $g$.  This follows from corollary \ref{main} (and proposition \ref{extrastruct}) noting that $X_n$ is $ M(n-1,1) =M(\lambda(n))$ for $\lambda=(1)$ in the notation of this paper. $\square$

\medskip

\noindent \textbf{Remark:} For a choice of $p$-adic integer $t$ we could construct these categories without the use of an ultrafilter by taking an appropriate direct limit of the categories $Rep(S_n)^{\le r}$, where we let $n$ approach $t$ in the $p$-adic topology sufficiently fast as $r$ tends to infinity gluing via our equivalences of categories.  This is similar in flavor to the construction of abelian envelopes $\mathcal{V}_t$ of the Deligne category $\underline{Rep}(GL_t)$ at integer $t$ appearing in an upcoming paper by Entova-Aizenbud, Hinich, and Serganova \cite{EHS} by taking an appropriate limit of subcategories of representations of supergroups.

\end{subsection}

\begin{subsection}{Modular representation theory in $p$-adic rank}\label{prank}

Deligne's category $\underline{\text{Rep}}(S_t)$ and its relatives are of interest to those studying tensor categories as they provide examples of pre-tannakian categories in characteristic zero which contain an object of super-exponential growth and hence do not admit a fiber functor (see \cite{EGNO} chapter 9.12).  The fact that they interpolate categories coming from representation theory allows one to do what Pavel Etingof calls `representation theory in complex rank' \cite{Etingof1}.

Unfortunately Deligne's construction of $\underline{\text{Rep}}(S_t)$ does not work very well in positive characteristic, it gives Karoubian categories which satisfy a certain universal property but which are very coarse and do not see much of the representation theory of the groups they are supposed to interpolate.  

So we'd like to introduce a program we are calling `Modular representation theory in $p$-adic rank'  as a modular version of this theory which is closer in behavior to the characteristic zero theory than the naive extension of Deligne's construction.  Here is an outline of the first construction in this direction based on the results in this paper:

\medskip

First we will define a monoidal category $\text{Rep}_\mathcal{R}(S_x)$ defined over the ring $\mathcal{R}$ of integer valued polynomials in a formal variable $x$ from section \ref{polysection}. It is defined as follows:

\begin{itemize}
\item Objects are symbols $[M(\lambda)]$ for all partitions $\lambda$, and formal direct sums thereof.

\item Morphisms are defined by $Hom([M(\lambda)], [M(\mu)]) := \overline{ \text{Hom}}(M(\mu), M(\lambda))\otimes \mathcal{R}$. In other words, it is the free $\mathcal{R}$ module spanned by stable tabloids.

\item Composition is given by the formula from Prososition \ref{polycomp}, with $p^\tau(n)$ getting replaced by $p^\tau(x) \in \mathcal{R}$.

\item The tensor structure is given by the stable tensor structure for permutation modules (see Prososition \ref{indep}).
\end{itemize}

Next, let $t$ be a $p$-adic integer and $k$ be a field of characteristic $p$. Define $\underline{\text{Rep}}_{k}(S_t)$  to be the Karoubian envelope of the reduction $\underline{\text{Rep}}^0_{k}(S_t)$ of $\text{Rep}_\mathcal{R}(S)$ under the map $ev_t: \mathcal{R} \rightarrow k$ from Prososition \ref{specr}.

\medskip

\noindent \textbf{Remark:} Maps from $\mathcal{R}$ into a field $k$ of characteristic zero are completely determined by the value $t \in k$ that the polynomial $x$ gets sent to.  One can check that in this setting the $k$-linear category we get by reducing $\text{Rep}_\mathcal{R}(S)$ under one of these maps and taking a Karoubian envelope is equivalent to Deligne's category $\underline{\text{Rep}}(S_t)$ provided $t\ne 0$.

\medskip

Let $\underline{\text{Rep}}_{k}(S_t)^{\le{r}}$ be the Karoubian subcategory generated by the images of $[M(\lambda)]$, this makes  $\underline{\text{Rep}}_{k}(S_t)$ into a filtered Karoubian tensor category. The real point though is that we set things up so that  $\underline{\text{Rep}}_{k}(S_t)^{\le{r}} \cong Young_n^{\le r}$ provided $n$ is sufficiently large and $(t-n)$ is sufficiently divisible by $p$.

We can embed these into abelian tensor categories using Delignes ultrafilter construction, or by gluing the filtered pieces appropriately.

\medskip

This program is still being developed, but there is a lot of partial progress. Here are a few of the initial goals of this program:

\begin{itemize}
\item Give more explicit presentations of the categories $\text{Rep}_\mathcal{R}(S_x)$ and $\underline{\text{Rep}}^0_{k}(S_t)$ in terms of a subset objects which generate them as tensor categories.

\item Define a more refined notion of categorical dimension which is better suited to positive characteristic. This is addressed in a subsequent paper \cite{EHO} with Etingof and Ostrik.

\item Find universal properties satisfied by the categories $\text{Rep}_\mathcal{R}(S_x)$ and $\underline{\text{Rep}}_{k}(S_t)$ analogous to the one for Deligne's categories.

\item Look into a classification of tensor functors from $\underline{\text{Rep}}_{k}(S_t)$ into pre-tannakian categories similar to the one given by Comes and Ostrik for Deligne categories at integer $t$ in \cite{CO2}.

\item Construct a version of this theory for modular representations of general linear groups, and other well behaved families of groups.

\item Understand a version of Schur-Weyl duality for these categories analogous to the results in \cite{Inna}.

\end{itemize}

\end{subsection}

\begin{subsection}{Miscellaneous open directions}
Finally we will close by mentioning a few questions we think are interesting related to this work:

\medskip



\medskip
\noindent \textbf{Putman's central stability:}  Putman defined a notion of central stability for a sequence of representations of symmetric groups, which is more relaxed than being an $FI$-module but still gives a number of stabilization results (see \cite{Putman}).  We suspect that a version of Proposition \ref{fiprop} holds for these sequences as well.

\medskip 

\noindent \textbf{Other groups:} Deligne's ultrafilter construction works with minimal changes to the sequences of groups $GL(n)$, $O(n)$, and $Sp(2n)$ along with their defining representations.  In his letter to Ostrik, Deligne suggests that a version of theorem \ref{Delconj} should hold for these groups as well. This would amount to proving periodicity results similar to those in this paper for these families of groups.

\medskip

\noindent \textbf{Degree functions on regular categories:}  In \cite{Knop} Knop gave a very general recipe for constructing interpolation categories from a regular category with a `degree' map, when applied to the opposite category of finite sets with an appropriate choice of degree map it recovers Deligne's category $\text{Rep}(S_t)$.  Is there a modified notion of degree map which is better suited to positive characteristic so that his construction can recover the categories $\underline{\text{Rep}}_{k}(S_t)^{\le{r}}$ from section \ref{prank}?

\medskip

\noindent \textbf{`Type 3' results:}  Our equivalences of categories let us equate a number of numerical invariants for representations of different symmetric groups (i.e. Prososition \ref{numeric}).  However it doesn't seem to help us actually calculate any of these values, many of which are notoriously difficult to understand.  The hope is that these stable (or stably periodic) values may be easier to compute than the usual versions.

\end{subsection}

\end{section}

\end{document}